\documentclass{amsart}

\usepackage{amssymb, amsfonts, amsmath, amsthm, textcomp}

\newtheorem{theorem}{Theorem}[section]

\newtheorem{corollary}[theorem]{Corollary}

\DeclareMathOperator{\diag}{diag}

\begin{document}

\title[The generalized Ces\`{a}ro matrices of order three are ....]{The generalized Ces\`{a}ro matrices of \\ order three are supraposinormal on $\ell^2$}

\author[H. C. Rhaly Jr.]{H. C. Rhaly Jr.}
\address{
1081 Buckley Drive\\
Jackson, MS   39206\\ 
U. S. A.}
\email {rhaly@alumni.virginia.edu}

\dedicatory{In memory of Douglas Bertrand Thorburn (1949-2015)}



\begin{abstract}
       The procedure that was used in an earlier paper on the generalized Ces\`{a}ro matrices of order two is adapted here to show that generalized Ces\`{a}ro matrices of order three are  supraposinormal on $\ell^2$.  This leads to information about their posinormality, coposinormality, and hyponormality.  The reader is then invited to formulate a conjecture regarding generalized Ces\`{a}ro matrices of other orders.       
        \end{abstract}

\maketitle

\section{Introduction}

The concept of supraposinormality for operators on a Hilbert space $H$ was introduced and investigated in \cite{2}.  The operator $A  \in \mathcal{B}(H)$, the set of bounded linear operators on $H$, is {\it supraposinormal} if there exist positive operators $P$ and $Q \in \mathcal{B}(H)$ such that  \[AQA^*= A^*PA,\] where at least one of $P$, $Q$ has dense range.  The ordered pair $(Q,P)$ is referred to as an \textit{interrupter pair} associated with $A$.  A \textit{posinormal} operator is a particular case of a supraposinormal with $Q=I$, and a \textit{coposinormal} operator is a particular case of a supraposinormal with $P=I$.   From \cite[Theorem $1$]{2} we know that if $P$ is invertible, then the supraposinormal operator $A$ is coposinormal; and we know that if $Q$ is invertible, then the supraposinormal operator $A$ is posinormal.

Recall that the operator $A \in \mathcal{B}(H)$ is \textit{hyponormal}  if \[< (A^*A - AA^*) f , f >   \geq  0\]  for all $f \in H$.   The following theorem will be of use in the next section.

\begin{theorem}   If $A$ is supraposinormal operator on $H$ with $AQA^*= A^*PA $ and  

\[Q \geq I \geq P \geq 0,\]   

\noindent then $A$ is hyponormal.
\end{theorem}

\begin{proof}   Assume the hypothesis is satisfied.  Then   \begin{align*}   \langle (A^*A - AA^*) f , f \rangle = \langle (A^*A - A^* P A + A Q A^* - AA^*) f , f \rangle   \end{align*} 
\begin{align*}   = \langle (I-P)Af,Af \rangle + \langle(Q-I)A^*f,A^*f \rangle  \geq 0   \end{align*}  

 \noindent for all f in $H$, so $A$ is hyponormal.
\end{proof}

As was noted in \cite{3}, the generalized Ces\`{a}ro matrices of order $\beta$, $(C^{(\alpha)}, \beta)$, for $\alpha, \beta \geq 0$ are the lower triangular infinite matrices given by
\begin{align*} (C^{(\alpha)}, \beta)_{ij} =  \frac{\Gamma(i+\alpha+1)}{\Gamma(i-j + 1)} \cdot  \frac{\beta  \Gamma(i-j + \beta)} {\Gamma(i + \alpha + \beta + 1)} . \end{align*}  For fixed $\alpha \geq 0$ and $\beta = 1$, one obtains the generalized Ces\`{a}ro matrices of order one, $(C^{(\alpha)},1)$, with entries \[m_{ij} =  (C^{(\alpha)},1)_{ij} = \begin{cases}
\frac{1}{i + 1 + \alpha} &  \text{ for }  0 \leq j \leq i \\
0 & \text{ for } j>i.
\end{cases}\]   for all $i$.   We note that $(C^{(\alpha)},1) \in \mathcal{B} (\ell^2)$ for all $\alpha > -1$, not just for $\alpha \geq 0$. The operators  $(C^{(\alpha)},1)$ are known to be supraposinormal, posinormal, and coposinormal for all $\alpha > -1$, and $(C^{(\alpha)},1)$ is hyponormal for all $\alpha \geq 0$.

 Using a diagonal operator $P$ and a ``nearly" diagonal operator $Q$,  it was shown in \cite{3} that the matrix $M : \equiv (C^{(\alpha)},2)$, the generalized  Ces\`{a}ro matrices of order two, with entries given by
\[m_{ij} = \begin{cases} 
\frac{2(i + 1 -j )}{(i+ 1 + \alpha)(i+ 2 + \alpha)}& \text{ for }  0 \leq j \leq i \\
0 & \text{ for }  j > i , \end{cases}\]   
are supraposinormal, posinormal, and coposinormal for $\alpha > -1$ and hyponormal for $\alpha \in {0} \cup [1,\infty)$.

\section{Main Results}

Under consideration here will be the matrix $M : \equiv (C^{(\alpha)},3)$ whose entries are as follows: 
\[m_{ij} = \begin{cases} 
\frac{3(i + 1 -j )(i + 2 -j )}{(i+ 1 + \alpha)(i+ 2 + \alpha)(i+ 3 + \alpha)}& \text{ for }  0 \leq j \leq i \\
0 & \text{ for }  j > i . \end{cases}\]   

\noindent We note that  $M \in \mathcal{B}(\ell^2)$ for all $\alpha > -1$.  The procedure followed below is similar to that used in \cite{3}, and some of the calculations have been aided by the Sagemath software system \cite{4}.

\begin{theorem}  The operator $M \in \mathcal{B}(\ell^2)$ is supraposinormal for $\alpha > -1$.  \end{theorem}

\begin{proof} We take  \begin{equation*} P :\equiv \diag \Big \{ \frac{(n + 1 + \alpha)(n + 2 + \alpha)(n + 3 + \alpha)}{(n + 4 + \alpha)(n + 5 + \alpha)(n + 6 + \alpha)} : n \geq 0 \Big \}, \end{equation*} and then compute  $M^*PM$. 

For $j \geq i$, the ($i$, $j$)$^{th}$ entry  of $M^*PM$ is \footnotesize  \begin{equation*} \sum_{k=0}^{\infty} \frac{9(j + 1 - i + k)(j + 2 - i + k)(k + 1)(k + 2)}{(j + 1+ k + \alpha)(j + 2 + k + \alpha)(j + 3 + k + \alpha)(j + 4 + k + \alpha)(j + 5+ k + \alpha)(j + 6+ k + \alpha)}. \end{equation*}

\noindent The series is telescoping, as can be seen by rewriting the summand as \begin{equation*} s(k) - s(k+1) \end{equation*} where  \begin{equation*} s(k) :\equiv   \frac{ak^4+bk^3+ck^2+dk+e}{(j + 1+ k +  \alpha)(j + 2+ k + \alpha)(j + 3+ k + \alpha)(j + 4 + k + \alpha)(j + 5+ k + \alpha)}  \end{equation*}
with \begin{equation*}  a=9, b = 9(8+2\alpha-i+3j),  c = 3 ( 71  -15i +i^2+57j-5ij+10j^2+42 \alpha +6 \alpha^2 -3 \alpha i + 15 \alpha j), \end{equation*}

\begin{equation*} d= \frac{3}{2} (180 - 48i+6i^2+ 236 j -36 ij +  i^2j +90j^2 - 5 ij^2 + 10 j^3 + 188 \alpha + 60 \alpha^2 + 6 \alpha^3 - 24 \alpha i - 3 \alpha^2 i  )   \end{equation*}

\begin{equation*}  + \frac{3}{2}(  \alpha i^2 +144 \alpha j + 21 \alpha^2j - 8 \alpha ij + 25 \alpha j^2)  , \end{equation*} and

\begin{equation*} e= \frac{3}{10} (4+j+\alpha)(5+j+\alpha)(20+24\alpha+6\alpha^2-6i-3\alpha i+i^2+30j+15\alpha j-5ij+10j^2) . \end{equation*}

\noindent Consequently, for $j \geq i$, the ($i$, $j$)$^{th}$ entry  of $M^*PM$ in simplified form is \begin{equation*} s(0) = \frac{3(20+24\alpha+6\alpha^2-6i-3\alpha i+i^2+30j+15\alpha j-5ij+10j^2)}{ 10  (j + 1+ \alpha)(j + 2 + \alpha)(j + 3 + \alpha)} .\end{equation*}  The expression for $i > j$ is similar, with the roles of $i$ and $j$ reversed.

Suppose that $Q = [q_{ij}]$ is given by 

\begin{equation*} Q = \left(\begin{array}{ccccccc} q_{00} & q_{01} & q_{02} &0&0&0&\ldots\\ q_{10} & q_{11} & q_{12} &0&0&0&\ldots\\q_{20}&q_{21}&q_{22}&0&0 &0 & \ldots\\0&0&0&1& 0&0 &\ldots\\0&0&0&0&1 &0 &\ldots\\0&0&0&0&0 &1 &\ldots\\\vdots&\vdots&\vdots&\vdots&\vdots&\vdots&\ddots \end{array}\right);    \end{equation*}   where  $q_{00} =  \frac{(10+12\alpha+3\alpha^2)(1+\alpha)(2+\alpha)(3+\alpha)}{60} $, $q_{10}=q_{01} = - \frac{\alpha(11+4\alpha)(1 + \alpha) (2 + \alpha)(3 + \alpha)}{40}$, 
 
  $q_{20}=q_{02}=   \frac{\alpha(3+2\alpha)(1 + \alpha) (2 + \alpha)(3 + \alpha)}{40}$, $q_{11} =\frac{(5-\alpha+15\alpha^2+6\alpha^3) (2 + \alpha)(3+\alpha)}{30}$, 
   
  $q_{12} = q_{21} = - \frac{\alpha(1 + \alpha) (2 + \alpha)(3 + \alpha)(1+4\alpha)}{40}$, $q_{22}=\frac{(20-6\alpha+7\alpha^2+6\alpha^3+3\alpha^4) (3+\alpha)}{60}$,  
  
  $q_{ii} =1$ for all $i=j \geq 3$, and $q_{ij} = 0$ otherwise.  
  
   \noindent If $X :\equiv QM^*$, then $X = [x_{ij}]$ has entries

\[ x_{ij} = \begin{cases}
\frac{(10j^2+(15\alpha+30)j+20+24\alpha+6\alpha^2)(1+\alpha)(2+\alpha)(3+\alpha)}{20(j+1 + \alpha )(j +2 + \alpha)(j+\alpha+3)} & \text{ for }  i=0 \\
\frac{((10-20\alpha)j^2+(10-50\alpha-30\alpha^2)j-33\alpha-45\alpha^2-12\alpha^3) (2 + \alpha) (3 + \alpha)}{20(j+1 + \alpha )(j +2 + \alpha)(j+\alpha+3)} & \text{ for }  i=1 \\
\frac{((20+10\alpha^2)j^2+(-20+30\alpha+35\alpha^2+15\alpha^3)j+18\alpha+39\alpha^2+27\alpha^3+6\alpha^4) (3 + \alpha)}{20(j+1 + \alpha )(j +2 + \alpha)(j+\alpha+3)} & \text{ for }  i=2 \\
\frac{3 (j+1-i)(j+2-i)}{(j+1+\alpha)(j+2+\alpha)(j+\alpha+3)} & \text{ for } j \geq i \geq 3 \\
0 & \text{ for } i \geq 3, j < i.  \end{cases}\]

\noindent For $j \geq i \geq 3$, the ($i$,$j$)$^{th}$ entry  of $MQM^*=MX$ is \footnotesize \begin{align*} \frac{3(i+1)(i+2)}{(i+1+\alpha)(i+2+\alpha)(i+3+\alpha)} \cdot x_{0j} \end{align*}   \begin{align*}  +  \frac{3i(i+1)}{(i+1+\alpha)(i+2+\alpha)(i+3+\alpha)} \cdot x_{1j}    \end{align*}    \begin{align*}  +  \frac{3(i-1)i}{(i+1+\alpha)(i+2+\alpha)(i+3+\alpha)} \cdot x_{2j}    \end{align*}    \begin{align*}  + \sum_{k=2}^{i-1} \frac{3(i-k)(i-k+1)}{(i+1+\alpha)(i+2+\alpha)(i+3+\alpha)} \cdot \frac{3(j-k)(j-k+1)}{(j+1+\alpha)(j+2+\alpha)(j+\alpha+3)}   \end{align*}  \begin{align*} =   \frac{3(20+24\alpha+6\alpha^2-6i-3\alpha i+i^2+30j+15\alpha j-5ij+10j^2)}{ 10  (j + 1+ \alpha)(j + 2 + \alpha)(j + 3 + \alpha)} ,  \end{align*} \normalsize  as is best verified with the aid of computational software.  The cases $j \geq i = 0$, $j \geq i = 1$,  and $j \geq i = 2$ are left to the reader. The computations for $i \geq j$ are similar, so it follows that $M^*PM = MQM^*$.  It is clear that $P$ has dense range, so $M$ is supraposinormal. 
\end{proof}

\begin{corollary}  The operator $M$ is both posinormal and coposinormal for \mbox{$\alpha > -1$.} \end{corollary}

\begin{proof} Since  \begin{equation*} \det \left(\begin{array}{cc} q_{00}&q_{01}  \\ q_{10}&q_{11}    \end{array}\right) = \frac{(1+\alpha)(2+ \alpha)^2(3+ \alpha)^2(4+\alpha)(20+ 15\alpha+3 \alpha^2)}{2880} \end{equation*}  and  \begin{equation*} \det \left(\begin{array}{ccc} q_{00}&q_{01} &q_{02}  \\ q_{10}&q_{11} &q_{12}  \\ q_{20}&q_{21} &q_{22}  \end{array}\right) = \frac{(1+\alpha)(2+\alpha)^2 (3+\alpha)^3(4+\alpha)^2(5+\alpha)}{8640}, \end{equation*} it is clear that $Q$ is positive and invertible for $\alpha > -1$.  Since $P$ is also clearly positive and invertible for $\alpha > -1$, it follows from \cite[Theorem $1$]{2}  that $M$ is both posinormal and coposinormal when $\alpha > -1$.   \end{proof}

\begin{corollary}  If $n$ is a positive integer and  $\alpha > -1$, then $M^n$ is both posinormal and coposinormal.
\end{corollary}

\begin{proof}  This follows from Corollary $2.2$ and \cite[Corollary $1$ (b)]{1}. 
\end{proof}

\begin{corollary}  The operator $M$ is hyponormal for $\alpha \in \{0,1\} \cup [2,\infty)$.
\end{corollary}

\begin{proof}
Since   \begin{equation*} \det \left(\begin{array}{cc} q_{00}-1&q_{01}  \\ q_{10}&q_{11}-1    \end{array}\right) = \frac{(\alpha-1)\alpha^2(1+\alpha)(2308+1860\alpha+521\alpha^2+60\alpha^3+3\alpha^4)}{2880} \end{equation*}    and    \begin{equation*} \det \left(\begin{array}{ccc} q_{00}-1&q_{01} &q_{02}  \\ q_{10}&q_{11}-1 &q_{12}  \\ q_{20}&q_{21} &q_{22}-1  \end{array}\right) = \frac{(\alpha-2)(\alpha-1)^2\alpha^3(1+\alpha)^2(2+\alpha)}{8640}, \end{equation*}  are nonnegative for $\alpha \in \{0,1\} \cup [2,\infty)$, it follows from Theorem $1.1$ that $M$ is hyponormal  for those values of $\alpha$.
\end{proof}

\section{Closing Remark}

On the basis of what is now known regarding supraposinormality, posinormality, coposinormality, and hyponormality for the generalized Ces\`{a}ro matrices of orders one, two, and three, the reader is invited to formulate a conjecture concerning these properties for the generalized Ces\`{a}ro matrices of order $k$ when $k \in \mathbb{Z}+$.

\end{document}